\documentclass[reqno]{amsproc}
\usepackage{tikz}
\usepackage{amstext,amsmath,amssymb,amsfonts}
\usepackage[latin1]{inputenc}
\usepackage{epsfig}
\usepackage{hyperref}
\usepackage{color}

\usepackage{geometry}
\usepackage{amsthm}

\definecolor{myurlcolor}{rgb}{0,0,0.7}

\newtheorem{rmq}{Remark}[section]
\newtheorem{dfn}{Definition}[section]
\newtheorem{lem}{Lemma}[section]
\newtheorem{thm}{Theorem}[section]

\newcommand{\bprof}{\begin{prof}}
\newcommand{\eprof}{\end{prof}}
\newenvironment{prof}[1][Proof]{\textbf{#1.} }{\ \rule{0.5em}{0.5em}}
\usepackage{cite}

\newcommand{\bea}{\begin{eqnarray}}
\newcommand{\eea}{\end{eqnarray}}
\newcommand{\beq}{\begin{equation}}
\newcommand{\eeq}{\end{equation}}
\newcommand{\enn}{\nonumber \end{equation}}
\newcommand{\beqs}{\begin{eqnarray*}}
\newcommand{\eeqs}{\end{eqnarray*}}

 \newcommand{\cE}{\mathcal{E}}
\newcommand{\cT}{\mathcal{T}}

 \newcommand{\R}{\mathbb{R}}

\def\cN{{\mathcal N}}

\title[RESIDUAL-BASED A POSTERIORI ERROR
ESTIMATES]
{RESIDUAL-BASED A POSTERIORI ERROR
	ESTIMATES FOR A CONFORMING MIXED FINITE ELEMENT
	DISCRETIZATION OF THE MONGE-AMP\`ERE EQUATION}

\author{Adetola Jamal $^{(a)}$}
\email{a) adetolajamal58@yahoo.com}
\address{Institut de Math\'ematiques et de Sciences Physiques (IMSP),
	Universit\'e d'Abomey-Calavi (UAC), Rep. of Benin}

\author{Hou\'edanou Koffi Wilfrid $^{(b)}$}
\email{b) khouedanou@yahoo.fr}
\address{
	D\'epartement de Math\'ematiques,
	Universit\'e d'Abomey-Calavi (UAC), Rep. of Benin
}

\author{Bernardin Ahounou $^{(c)}$}
\email{c) bahounou@yahoo.fr}
\address{
	D\'epartement de Math\'ematiques,
	Universit\'e d'Abomey-Calavi (UAC), Rep. of Benin
}

\begin{document}

\maketitle
\begin{Large}
\begin{abstract}\Large
	In this paper we develop a new a posteriori error analysis 
	for the Monge-Amp\`ere equation 
	approximated by conforming finite element method on 
	isotropic meshes in $\mathbb{R}^2$. The approach utilizes a slight variant of the 
	mixed discretization proposed by  Gerard Awanou and 
	Hengguang Li in \cite{GL:2014}.
	The a posteriori error estimate is based on a suitable evaluation on the residual of the finite element
	solution. It is proven that the a posteriori error estimate provided in this paper is both reliable and efficient.
	\\
	\small{\bf Mathematics Subject Classification [MSC]:} 74S05,74S10,74S15,
	74S20,74S25,74S30.
\end{abstract}

\tableofcontents
\section{General introduction}
The adaptive techniques have become indispensable tools
and unavoidable in the field of study behavior of the error committed during solving partial differential equations 
(PDE). 
A posteriori error estimators are computable quantities, expressed in terms of the discrete solution  and of the data that measure the actual discrete errors without the knowledge of the exact solution. They are essential to design adaptive mesh  refinement  algorithms  which equi-distribute the computational effort and optimize the approximation efficiency. 
Since the pioneering work of Babuska and Rheinboldt \cite{BR78b,BR78a},   adaptive finite element methods based on a posteriori error estimates have been extensively investigated. 
Several a-posteriori error analysis methods for  PDE have been developped in 
the last five decades \cite{TS:2001, SCS:1999, Adams:1978, AO:1997, Babuska:1978', BR:1993,BHV:2009, Brezis:1983, BRR:1980, ZZ:1988,46,47,AHN:15,HA:2016}.

We consider the Monge-Amp\`ere equation on a convex domain of $\mathbb{R}^2$ with a smooth solution and 
our approach utilizes a slight variant of the 
mixed discretization proposed by  Gerard Awanou and 
Hengguang Li in \cite{GL:2014}.
The  purpose of these work is to determine to which extend the general framework for 
adaptivity for nonlinear problems of Gatica and his collaborators  in \cite{CGOS:2015,47}
can be applied to the Monge-Amp\`ere equation.
More precisely, we attempt to determine to which extent one can prove results analogous 
to the ones of Hou\'edanou, Adetola and Ahounou \cite{HJA:2017}. Ideed, in \cite{GL:2014} Gerard Awanou and 
Li study the mixed method for this equation and they gave a priori error estimator under the assumption
of regularity for the solution of continuous problem. 
Omar Lakkis in his presentation of July 15, 2014, presents a family of reliable error indicators for a primal formulation \cite{OmarLakkis:2014}.
However, to our best knowledge, they din't talk about adaptative method
for this mixed formulation. In this case we have for main to give a posteriori error analysis  by 
constucting  reliable and efficiency indicator error.

In \cite{GL:2014} Gerard and Li
have  introduce  a mixed finite element method formulation for the  elliptic Monge-Amp\`ere equation by
puting $\sigma = D^2u$. The news unknowns in the formulation are
$u$ and $\sigma$ which been  approached respectively by the discrete  polynomials spaces of Lagrange.
Finally, they give a result of error priori analysis with some numerical tests confirming the convergence rates.
In this paper, we have got a new family of a local indicator error  $\Theta_{K}$ (see Definition \ref{definitionindicator}, eq. \ref{localindicator}) and global $\Theta$ (eq. \ref{globalindicator}) efficiency and reliability for
the mixed method  of  Monge-Amp\`ere model. 
 We prove that our indicators error 
are efficiency and reliability, and then, are optimal.  The global inf-sup condition is the main tool yielding the reliability. In turn,
The local efficiency result is derived using the technique of bubble 
function introduced by R. Verf\"{u}rth \cite{40} and used in similar context by 
C. Carstensen \cite{25,ca:97}.

The   paper  is organized as follows.
Some  preliminaries and  notation are given in  section  \ref{sec:R1}. 
In section \ref{sec:R2}, the a posteriori error estimates are derived. We offer our conclusion and the further works in Section \ref{summary}.

\section{Preliminaries and Notation}\label{sec:R1}
\subsection{Model}
Let $\Omega$ be a convex polygonal  domain
of $\R^2$  with boundary $\partial \Omega$. We consider the following problem : find the unique 
strictly convex $C^3(\overline{\Omega})$ solution $u$ (when it exists) of 
\begin{align} \label{m1}
\begin{split}
\det (D^2 u) & = f \, \text{in} \, \Omega \\
u & = g \, \text{on} \, \partial \Omega.
\end{split}
\end{align}
The given function $f \in C^1(\overline{\Omega})$  is assumed to satisfy $f  >0$
and the function $g\in C(\partial \Omega)$ is also given and
assumed to extend to a $C^3(\overline{\Omega})$ function. Here $\det (D^2 u)$ denotes the determinant of the Hessian matrix
$D^2 u=\big( \partial^2 u/(\partial x_i \partial x_j)\big)_{i,j=1, 2} $.
\subsection{Modified continuous mixed weak formulation}
We begin this subsection by introducing some useful notations. If $W$ is a bounded domain of $\mathbb{R}^2$ and $m$ 
is a non negative integer, the Sobolev space $H^m(W)=W^{m,2}(W)$ is 
defined in the usual way with the usual norm $\parallel\cdot\parallel_{m,W}$ and semi-norm $|.|_{m,W}$. In particular, 
$H^0(W)=L^2(W)$ and we write $\parallel\cdot\parallel_W$ for $\parallel\cdot\parallel_{0,W}$.
Similarly we   denote by
$(\cdot,\cdot)_{W}$  the $L^2(W)$, $[L^2(W)]^2$ or $ [L^2(W)]^{2\times 2}$ inner product.
Now, we recall the continuous mixed weak formulation introduce by G\'erard et al. \cite{GL:2014}.
The mixed  weak formulation of \eqref{m1}  is : find $(\sigma,u) \in  H^1(\Omega)^{2 \times 2}\times H^2(\Omega)$ 
such that
\begin{equation} 
\left \{
\begin {array}{ccl}\label{m10}
(\sigma,\mu)_{\Omega} + (\nabla\cdot\mu, D u)_{\Omega} - < D u, \mu n >_{\partial \Omega} & =& 0,
\,\, \forall \mu \in H^1(\Omega)^{2 \times 2} \\ \\
(\det \sigma,v) & = &(f,v)_{\Omega}, \,\, \forall v \in H_0^1(\Omega)\\ \\
u & =& g \,\, \text{on} \, \partial \Omega.
\end{array}
\right.      
\end{equation}

Let introduce Lagrange multiplier $\lambda= u|_{\partial \Omega}$, 
the  modified mixed  weak formulation of \eqref{m1}  is : find $(\sigma,u, \lambda) \in 
H^1 (\Omega)^{2 \times 2}\times H^2(\Omega) \times  H^{1/2}(\partial \Omega) $ 
such that
\begin{equation} 
\left \{
\begin {array}{ccl}\label{m100}
(\sigma,\mu)_{\Omega} + (\nabla\cdot \mu, D u)_{\Omega} - \left<D u, \mu n \right>_{\partial \Omega} & =& 0,
\,\, \forall \mu \in H^1(\Omega)^{2 \times 2} \\ \\
(\det \sigma,v)_{\Omega} & = &(f,v)_{\Omega}, \,\, \forall v \in H_0^1(\Omega)\\ \\
(\lambda, \mu)_{\partial \Omega} & =& (g, \mu)_{\partial \Omega} \,\, \forall \mu \in H^{-1/2}( \partial \Omega).
\end{array}
\right.      
\end{equation}
\begin{lem} (ref. \cite{GL:2014})
	The problem  \eqref{m10} is well defined, and  if $u$ is a smooth solution 
	of \eqref{m1}, then $(u, D^2 u)$ solves \eqref{m10}.
\end{lem}
\begin{rmq}
	In this formulation the boundary condition $u=g$ on $\partial \Omega$ viewed as constraints and
	imposed via Lagrange multiplier.
\end{rmq}

We end this section with some notation. Let $\mathbb{P}_k$ be the space of polynomials of total degree not larger than $k$. In order to avoid excessive use of constants, the abbreviations $x\lesssim y$ and $x\sim y$ stand for $x\leq cy$ and $c_1x\leq y\leq c_2x$, respectively, with positive constants independent of $x$, $y$ or $\cT_h$ (meshes). 
\subsection{Modified discrete formulation}
Let $\Omega$ be an open convex bounded subset of $\mathbb{R}^2$ with boundary $\partial \Omega$ and let 
$\cT_h$ denote a triangulation of $\Omega$ into simplices $K$. We denote by $h_K$ the diameter of the element $K$ and $h=\displaystyle\max_{K\in\cT_h} h_K$. We make the assumption that the triangulation is conforming and satisfies the usual shape regularity condition, i.e. there exists a constant $\sigma> 0$ suth that:
 $\frac{h_K}{\rho_K}\leq \sigma,$ for all $K\in \cT_h$ where $\rho_K$ denotes the radius of the  largest ball inside $K$. (See Figs. \ref{isotropic}, \ref{adm1},\ref{adm2}).
 \begin{figure}[http]
 	\begin{minipage}[c]{.30\textwidth}
 		\centering
 		\begin{center}
 			\begin{tikzpicture}[scale=0.5]
 			\draw (0,1)--(7,1);
 			\draw (0,1)--(2,4.5);
 			\draw (7,1)--(2,4.5);
 			\draw [line width=0.75pt] [>=latex,<->](0,0.75)--(7,0.75)node [midway,below,sloped] {$\mbox{diam} (K)=h_K$};
 			\draw  (2.45,2.43) circle (1.4);
 			\draw (2.45,2.) node [above]{$\bullet$};
 			\draw (3,1.5) node [above]{\small{$\rho_K$}};
 			\draw  [line width=0.75pt](2.45,2.43)--(2.5,1);
 			\end{tikzpicture}
 		\end{center}
 		\caption{\footnotesize{\small\small{Isotropic element $K$ in $\mathbb{R}^2$.}}}
 		\label{isotropic}
 	\end{minipage}
 	\begin{minipage}[c]{.30\textwidth}
 		\centering
 		\begin{center}
 			\begin{tikzpicture}[scale=0.5]
 			\draw (0,3)--(0,-3);
 			\draw (0,3)--(2,0);
 			\draw (0,3)--(-2,0);
 			\draw (-2,0)--(2,0);
 			\draw (2,0)--(0,-3);
 			\draw (-2,0)--(0,-3);
 			\end{tikzpicture}
 		\end{center}
 		\caption{\footnotesize{\small\small Example of conforming mesh in $\mathbb{R}^2$}}
 		\label{adm1}
 	\end{minipage}
 	\begin{minipage}[c]{.30\textwidth}
 		\centering
 		\begin{center}
 			\begin{tikzpicture}[scale=0.5]
 			\draw (0,3)--(0,0);
 			\draw (0,3)--(2,0);
 			\draw (0,3)--(-2,0);
 			\draw (-2,0)--(2,0);
 			\draw (2,0)--(0,-3);
 			\draw (-2,0)--(0,-3);
 			\draw (0,0)node {$\bullet$};
 			\end{tikzpicture}
 		\end{center}
 		\caption{\footnotesize{\small\small Example of nonconforming mesh in $\mathbb{R}^2$ }}
 		\label{adm2}
 	\end{minipage}
 \end{figure}

 To be able to use global inverse estimates, c.f. (2.2) and $(2.3)$ of \cite{GL:2014}, we require the triangulation to be also quasi-uniform, i.e. there is a constant $C> 0$ such that $h\leq Ch_K$ for all $K\in\cT_h$. 
 
 For any $K\in \mathcal{T}_h$, we denote by $\cE_h (K)$ (resp. $\cN_h(K))$
 the set of its edges (resp. vertices)  and set 
 $\cE_h=\displaystyle\bigcup_{K\in\mathcal{T}_h} \cE(K)$, $\cN_h=\displaystyle\bigcup_{K\in\mathcal{T}_h} \cN(K)$. 
 For $\mathcal{A}\subset \overline{\Omega}$ we define: 
 $$
 \cE_h(\mathcal{A})=\left\{ E\in\cE_h: E\subset \mathcal{A}\right\}.
 $$
 
 Let $V_h$ denote the standard Lagrange finite element space of degree $k\geq 3$ and $\Sigma_h=V_h^{2\times 2}$; that is we consider the following discret spaces :
 \begin{eqnarray}
 V_h := \{  v_h \in C^0(\overline{\Omega}) : v_{|K} \in \mathbb{P}_k(K) \quad \forall K \in \mathcal{T}_h, \quad k\ge 3\},
 \end{eqnarray}
 \begin{eqnarray}
 \Sigma_h := \{  \tau_h \in \left[C^0(\overline{\Omega})\right]^{2\times 2} 
 : \tau_{|K} \in [\mathbb{P}_k(K)]^{2\times 2} \quad \forall K \in \mathcal{T}_h, \quad k\ge 3\}
 \end{eqnarray}
 The discrete  formulation of \eqref{m10} is given by :
 find $(u_h,\sigma_h) \in V_h\times \Sigma_h$ such that
 \begin{equation}\label{m11}
 \left \{
 \begin {array}{ccl}
 (\sigma_h , \tau_h)+ (\nabla\cdot \tau_h, Du_h)-\langle Du_h , \tau_h n \rangle&=&0
 \quad \forall \tau_h \in  \Sigma_h \\ \\
 (det\, \sigma_h , v_h) &=&(f,v_h) \quad \forall v_h \in V_h    \\ \\
 u_h &=& g_h \quad  \textrm{on}\quad  \partial{\Omega},  
 \end{array}
 \right.
 \end{equation}

  We recall that $H_0^1(\Omega)$ is the subset of $H^1(\Omega)$ of elements with vanishing trace on $\partial \Omega$. Let $I_h$ denote the standard Lagrangian interpolation operator from $C^0(\partial\Omega)$ into the space $L_h:=\{{v_h}_{|\partial\Omega}: v_h\in V_h\}$. 
The modified discrete formulation of \eqref{m100} is given by  find : For $\rho>0$, we difine by :
$(\sigma_h,u_h,\lambda_h) \in V_h\times \Sigma_h \times L_h$ such that
\begin{equation}\label{m111}
\left \{
\begin {array}{ccl}
(\sigma_h , \tau_h)_{\Omega}+ (\nabla\cdot \tau_h, Du_h)_{\Omega}-\langle Du_h , \tau_h n \rangle_{\partial \Omega}&=&0
\quad \forall \tau_h \in  \Sigma_h \\ \\
(det\, \sigma_h , v_h)_{\Omega} &=&(f,v_h)_{\Omega} \quad \forall v_h \in V_h    \\ \\
\langle \lambda_h, \mu\rangle_{\partial \Omega} &=& \langle g_h, \mu \rangle_{\partial \Omega} \quad  \forall \mu \in L_h,
\end{array}
\right.
\end{equation}
with $$g_h=I_hg.$$ 
 Now, we define
\begin{equation*}
\bar{B}_h(\rho) =\left \{ (\eta_h,w_h) \in \Sigma_h\times V_h, \|w_h-I_hu\|_{H^1} \le \rho,
\|\eta_h-I_h\sigma\|_{L^2} \le h^{-1}\rho 
\right \}
\end{equation*}
and
\begin{equation*}
{B}_h(\rho) =  \bar{B}_h(\rho) \cap Z_h,
\end{equation*}
where 
\begin{align*}
Z_h:=\{ (w_h,\eta_h) \in H_h\times Q_h, w_h=g_h \,\, \mbox{on}\,\, \partial \Omega,\\
(\eta_h,\tau_h)+(\nabla\cdot \tau_h, Dw_h)-\langle Dw_h,\tau_h n\rangle =0\quad \forall \tau_h\in Q_h\}.
\end{align*}
By simple calcultions, the problem  \eqref{m111} is logically equivalent to \eqref{m11} and we have the result:
\begin{lem} (cf. \cite{GL:2014})
	The problem  \eqref{m111} as unique  solution in ${B}_h(\rho)$.
\end{lem}
\begin{thm}\cite{GL:2014}
	Let $ (u,\sigma) \in  H^{k+1}(\Omega) \times H^{k}(\Omega)^{2\times 2}$ be  the convex solution 
	of non linear problem \eqref{m10} with $k\ge 3$. The discrete non linear problem  \eqref{m11}
	has a unique solution $(u_h,\sigma_h)$ in $\mathcal{B}_h(\rho) \subset V_h \times \Sigma_h$.
	Moreover the following estimate holds.  
	\begin{eqnarray}
	\| u-u_h\|_{H^1} \leq Ch^k. 
	\end{eqnarray}
	\begin{equation}
	\| \sigma-\sigma_h\|_{L^2} \leq  C h^{k-1}.
	\end{equation}
\end{thm}
\section{A-posteriori error analysis}\label{sec:R2}
In order to  solve Monge-Amp\`ere problem by efficient adaptive finite element methods, reliable and efficient 
a posteriori error analysis is important to provide appropriated indicators. 
In this section, we first define the local and global indicators and then the lower and upper error  bounds are derived.
\subsection{Residual Error Estimators}
The general philosophy of residual error estimators is to estimate an appropriate norm of the correct residual by terms that 
can be  evaluated easier, and that involve the data at hand.

Let $\textbf{U}=(\sigma,u,\lambda)$ and  $\textbf{W}=(\tau,v,\mu)$. We define the operator $\textbf{B}$ by
\begin{equation*}
\langle\textbf{B}(\textbf{U}),\textbf{W}\rangle:=(\sigma,\tau)_{\Omega} + (\nabla\cdot \tau, D u)_{\Omega} 
- < D u, \tau n >_{\partial \Omega} -(\det \sigma,v)_{\Omega} + \langle \lambda, \mu \rangle _{\partial \Omega},
\end{equation*}
and 
\begin{equation*}
\langle \mathcal{F}, \textbf{W}\rangle:=(-f,v)_{\Omega}+ (g,\mu)_{\partial \Omega}.
\end{equation*}

We also define  by $\mathbf{H}= H^1 (\Omega)^{2 \times 2}\times H^2(\Omega) \times  H^{1/2}(\partial \Omega)$ and
$\textbf{M}=H^1 (\Omega)^{2 \times 2}\times H^1(\Omega) \times  H^{1/2}(\partial \Omega)$.
Then the  continuous problem \eqref{m100} is equivalent to :
find $\textbf{U}\in \mathbf{H}$, such that 
\begin{equation}
\langle\textbf{B}(\textbf{U}),\textbf{W}\rangle= \langle\mathcal{F}, \textbf{W}\rangle,\quad 
\forall \,\, \textbf{W}\in \textbf{M}
\end{equation}
We define the discrete version by the same way.

Then let, 
$\textbf{U}_h=(\sigma_h,u_h,\lambda_h)$ and  $\textbf{W}=(\tau,v,\mu)$. We define  by:
\begin{eqnarray}\nonumber
\langle\textbf{B}(\textbf{U}_h),\textbf{W}\rangle&=&(\sigma_h,\mu)_{\Omega} + (\nabla\cdot \mu, D u_h)_{\Omega} - 
<D u_h, \mu n >_{\partial \Omega} \\
&+&(\det \sigma_h,v)+ \langle \lambda_h, \mu \rangle _{\partial \Omega}
\end{eqnarray}
and
\begin{equation}
\langle\mathcal{F}, \textbf{W}\rangle:=(f,v)_{\Omega}+  (g_h,\mu)_{\partial \Omega}.
\end{equation}

We also define by $ \textbf{H}_h= V_h\times \Sigma_h \times L_h$. 
The discrete  problem \eqref{m111} is 
equivalent \\to :
find $\textbf{U}_h\in \textbf{H}_h$, such that, 
\begin{equation}
\langle\textbf{B}(\textbf{U}_h),\textbf{W}\rangle= \langle\mathcal{F}, \textbf{W}\rangle,
\quad \forall \,\,\textbf{W}  \in  \textbf{H}_h.
\end{equation}
We recall the following Lemma 
\begin{lem}(cf. \cite[Section 3.1]{GL:2014})\label{derivatedetermineant}
	\textbf{: Fr\'echet derivate of the determinant}. For $F(v)= \det D^2v$, we have $F'(v)(u)=(cof D^2v):D^2u$.
\end{lem}
By using the Lemma \ref{derivatedetermineant}, we have
the operator  $\mathbf{B}$ is differentiable and for all $\textbf{U}=(\sigma,u,\lambda)\in\textbf{H}$,  $\textbf{W}=(\tau,v,\mu)\in\textbf{M}$ and
$\textbf{V}=(\tau',v',\mu')\in\textbf{M}$, its  differential at $\textbf{V}=(\tau',v',\mu')$  is given 
by:
\begin{eqnarray}\nonumber
\langle D_{\mathbf{V}}\mathbf{B}(\mathbf{U}),\mathbf{W}\rangle &=&  
(\sigma,\tau)_{\Omega} + (\nabla\cdot \tau, D u)_{\Omega} 
- < D u, \tau n >_{\partial \Omega} \\
&+&(cof D^2v':D^2u,v) + \langle \lambda, \mu \rangle _{\partial \Omega}.
\end{eqnarray}
We deduce the existence of a positive constant $C_m$ , independent  of $\xi\in\textbf{H}$
 and the continuous and discrete solutions, such that the following global inf-sup condition holds:
\begin{eqnarray}\label{infsup}
C_m \lVert \mathbf{\xi} \lVert_{\textbf{H}} &\leq& \sup_{\mathbf{W}\in \mathbf{M}-\{\textbf{0}\} }
\frac{|\left<D\textbf{B}_{\textbf{V}}(\mathbf{\xi}),\textbf{W}\right>|}
{\|\mathbf{W} \|_{\textbf{M}}}
\end{eqnarray}

We have the following lemma:
\begin{lem}\label{firstestimate}
	There holds 
	\begin{eqnarray}
	\lVert \textbf{U}- \textbf{U}_h\lVert_{\textbf{H}} &\leq& C \lVert R \lVert_{\textbf{H}'}+\lVert f-\det \sigma_h-A:D^2u_h \lVert_{L^2(\Omega)}.
	\end{eqnarray}
	Where  $R$  is the residual functional define by $R(\textbf{W})= [\mathcal{F}-\mathbf{B}(\mathbf{U}_h), \mathbf{W}]$
	$\quad \forall \textbf{W} \in \mathbf{M}$,
	which satisfies: $$R(\mathbf{W}_h)=0\quad \forall \mathbf{W}_h \in \textbf{H}_h .$$
\end{lem}
\begin{proof}
	Let $\textbf{U}=(\sigma,u,\lambda)$,  $\textbf{W}=(\tau,v,\mu)$ and
	$\textbf{V}=(\tau',v',\mu')$, we have
	\begin{eqnarray*}
	\langle D_{\mathbf{V}}\mathbf{B}(\mathbf{U}),\mathbf{W}\rangle& = & 
	(\sigma,\tau)_{\Omega} + (\nabla\cdot \tau, D u)_{\Omega} \\
	&-& <D u, \tau n >_{\partial \Omega} -(cof D^2v':D^2u,v)
	+ \langle \lambda, \mu \rangle _{\partial \Omega}
	\end{eqnarray*}
and
	\begin{eqnarray*}
	\langle D_{\mathbf{V}}\mathbf{B}(\mathbf{U}-\textbf{U}_h),\mathbf{W}\rangle &=& 
	(\sigma,\tau)_{\Omega} + (\nabla\cdot \tau, D u)_{\Omega} \\
	&-& <D u, \tau n >_{\partial \Omega} -(cof D^2v':D^2u,v)\\ 
	&+& \langle \lambda, \mu \rangle _{\partial \Omega}\\ 
	&-& 
	[(\sigma_h,\tau)_{\Omega} + (\nabla\cdot \tau, D u_h)_{\Omega} \\
	&-& < D u_h, \tau n >_{\partial \Omega} +(cof D^2v':D^2u_h,v) \\
	&+& \langle \lambda_h, \mu \rangle_{\partial \Omega}] 
	\end{eqnarray*}
	Particulary for $\textbf{V}=\textbf{U}$, we have
	\begin{eqnarray*}
	\langle D_{\mathbf{V}}\mathbf{B}(\mathbf{U}-\textbf{U}_h),\mathbf{W}\rangle& = & \langle \mathcal{F},\mathbf{W} \rangle
	-[(\sigma_h,\tau)_{\Omega} \\
	&+&
	 (\nabla\cdot \tau, D u_h)_{\Omega} 
	- < D u_h, \tau n >_{\partial \Omega} +(A:D^2u_h,v) \\
	& +& (\det \sigma_h,v)-(\det \sigma_h,v)+(\det \sigma,v)+\langle \lambda_h, \mu \rangle _{\partial \Omega}],
	\end{eqnarray*}
	where  $A:= cof D^2u $. Therefore, 
	\begin{eqnarray}\label{residuequation}
	\nonumber
	\langle D_{\mathbf{V}}\mathbf{B}(\mathbf{U}-\textbf{U}_h),\mathbf{W}\rangle&=&
	\langle \mathcal{F}- \mathbf{B}(\mathbf{U}_h),\mathbf{W} \rangle+(f-\det \sigma_h,v)-(A:D^2u_h,v)\\
	\nonumber
	&=&\langle \mathcal{F}- \mathbf{B}(\mathbf{U}_h),\mathbf{W} \rangle+(f-\det \sigma_h-A:D^2u_h,v) \\
	&=& R(\mathbf{W})+(f-\det \sigma_h-A:D^2u_h,v)
	\end{eqnarray}
	Using Cauchy-Schwarz inequality and the inequality \eqref{infsup}, the result follows.
\end{proof}

Now we define the residual equation:
\begin{equation}
R(\mathbf{W})=[\mathcal{F}-\textbf{B}(\mathbf{U}_h), \mathbf{W}], 
\end{equation}
hence,
\begin{eqnarray*}
R(\mathbf{W})=&=&(f,v)_{\Omega}+(g,\mu)_{\partial \Omega}- (\sigma_h,\tau)_{\Omega}-(\nabla \tau, Du_h)_{\Omega}+ \langle \tau n, Du_h\rangle\\
&-& ( A:D^2u_h, v)_{\Omega}
-(\lambda_h,\mu)_{\partial \Omega}\\
&=&(f- A:D^2u_h,v)_{\Omega}- (\sigma_h,\tau)_{\Omega}-(\nabla \tau, Du_h)_{\Omega}\\
&+&\langle \tau n, Du_h\rangle_{\partial\Omega}
+(g-\lambda_h,\mu)_{\partial \Omega}.
\end{eqnarray*}
By integrating by parts, we obtain for $\textbf{W}=(v,\tau,\mu)\in\textbf{M}$, the equation:
\begin{equation}
R(\mathbf{W})=R_1(v)+R_2(\tau)+ R_3(\mu), 
\end{equation}
where: 
\begin{eqnarray*}
R_1(v)&:=&(f-A:D^2u_h,v)_{\Omega}, \mbox{     }     v\in \mathcal{H}:=H^1(\Omega)\\
R_2(\tau)&:=&(D^2u_h-\sigma_h,\tau)_{\Omega},\mbox{     } \tau \in\Sigma:=[H^1(\Omega)]^{2\times 2}\\
R_3(\mu)&:=&(g-\lambda_h,\mu)_{\partial \Omega}, \mbox{    }\mu\in H^{1/2}(\partial\Omega)
\end{eqnarray*}
In this way, it follows that:
\begin{equation}\label{relation01}
\lVert R \lVert_{\textbf{M}'} \leq C\{  \lVert R_1 \lVert_{\mathcal{H}'}+ \lVert R_2 \lVert_{\Sigma'} + \lVert R_3 \lVert_{H^{-1/2}(\partial \Omega)} \}
\end{equation}
and hence our next purpose is to derive suitable upper bounds for each one of the terms on the right hand side of (\ref{relation01}).
We start with the following lemma, which is a direct consequence of the Cauchy-Schwarz inequality.
\begin{lem}
	There exist $C_1$ , $C_2$ and $C_3 > 0$, independent of the meshsizes, such that:
	\begin{equation}\label{estimR1}
	\lVert R_1 \lVert_{\mathcal{H}'} \leq C_1 \left\{\displaystyle \sum_{K\in \mathcal{T}_h}
	\left( \lVert f_h-A:D^2u_h \lVert_K^2+ \lVert f-f_h \lVert_K^2\right)\right\}^{1/2}
	\end{equation}
	and 
	\begin{equation}\label{estimR2}
	\lVert R_2 \lVert_{\Sigma'} \leq C_2\left\{ \displaystyle \sum_{K\in \mathcal{T}_h}
	\lVert \sigma_h-D^2u_h  \lVert_K^2\right\}^{1/2}.
	\end{equation}
	In addition there holds
	\begin{equation}\label{estimR3}
	\lVert R_3 \lVert_{H^{-1/2}(\Omega)} \leq C_3\left\{\displaystyle \sum_{E\in \cE_h(\partial\Omega)}
	{\lVert g-\lambda_h\lVert_E^2}\right\}^{1/2}.
	\end{equation}
\end{lem}
\begin{lem}
	There exist a positive constant $C$, such that
	\begin{multline}\label{estimR}
	\lVert R \lVert_{\textbf{M}'} \leq C \left\{\left(\displaystyle \sum_{K\in \mathcal{T}_h}
	\lVert f_h-A:D^2u_h \lVert_K^2\right)^{1/2} +
	\left(\displaystyle \sum_{K\in \mathcal{T}_h}
	\lVert \sigma_h-D^2u_h  \lVert_K^2\right)^{1/2}+\right.\\
	\left.
	\left(\displaystyle \sum_{E\in \cE_h(\partial\Omega)}
	\lVert g-\lambda_h\lVert^2_{ E}\right)^{1/2}+
	\left(\displaystyle \sum_{K\in \mathcal{T}_h} \lVert f_h-f \lVert_K^2\right)^{1/2}\right\}.
	\end{multline}
\end{lem}
\begin{proof}
	By using \eqref{relation01},\eqref{estimR1},\eqref{estimR2},\eqref{estimR3} and Cauchy-Schwarz inequality,
	the result follows.
\end{proof}
\subsubsection{A posteriori error indicators}
Now, we define the error indicators:
\begin{dfn}[\textbf{A-posteriori error indicators}]\label{definitionindicator}
	Let $\text{U}_h=(\sigma_h,u_h,\lambda_h)\in\textbf{H}_h$ be the finite element solution. Then,
	the residual error estimator is locally defined by:
	\begin{eqnarray}\nonumber
		\Theta_K^2 (\textbf{U}_h)&:=&\lVert f_h-A:D^2u_h \lVert_K^2+\lVert \sigma_h-D^2u_h  \lVert_K^2\\
		\label{localindicator}
		&+& \lVert f_h-\det \sigma_h-A:D^2u_h\lVert_K^2\\\nonumber
		&+&\displaystyle\sum_{E\in\cE_h(K\cap\partial \Omega)}	\lVert g_h-\lambda_h\lVert^2_{E}.
	\end{eqnarray}
	
\end{dfn}
The global residual error estimator is given by: 
\begin{eqnarray}\label{globalindicator}
\Theta (\textbf{U}_h)&:=&\left(\sum_{K\in\mathcal{T}_h}\Theta_K^2 (\textbf{U}_h)\right)^{1/2}.
\end{eqnarray}
Furthermore denote the local and global approximation terms by:
\begin{eqnarray*}
	\zeta_{K}^2:= \lVert f-f_h \lVert_K^2+\displaystyle\sum_{E\in\cE_h(K\cap\partial \Omega)}\lVert g-g_h \lVert_E^2,
 \end{eqnarray*}
and 
\begin{eqnarray*}
	\zeta:=\left(\sum_{K\in\mathcal{T}_h}\zeta_K^2\right)^{1/2}.
\end{eqnarray*}

\subsubsection{Reliability of $\Theta$}
The first main result is given by the following theorem. 
\begin{thm}[\textbf{Reliability of the a posteriori error estimator}] Let $\textbf{U}=(\sigma,u,\lambda)\in\textbf{H}$ be the exact solution and $\textbf{U}_h=(\sigma_h,u_h,\lambda_h)\in\textbf{H}_h$ be the finite element solution. Then,
	there exist a positive constant $C$ such that:
	\begin{eqnarray}
		\lVert \mathbf{U}-\mathbf{U}_h \lVert_{\textbf{H}} &\leq& C\left[\Theta (\textbf{U}_h)+\zeta\right].
	\end{eqnarray}
\end{thm}

\begin{proof}
	By using Lemma \ref{firstestimate} and the estimate \eqref{estimR}, the result follows. 
\end{proof}
\subsubsection{Efficiency of $\Theta$}\label{efficiencysection}
The second main result of this paper is the efficiency of the a posteriori error estimator $\Theta$.
In order to derive the local lower bounds, 
we proceed similarly as in \cite{ca:97, HJA:2017,SGR:2016} (see also \cite{15}), by applying
inverse inequalities, and the localization technique based on simplex-bubble and face-bubble functions. To this end, we 
recall some notation and introduce further preliminary results. Given $K\in \mathcal{T}_h$, and 
$E\in \cE_h(K)$,
we let $b_K$ and $b_E$ be the usual simplexe-bubble and face-bubble 
functions respectively (see (1.5) and (1.6) in \cite{verfurth:96b}). In particular, $b_K$ satisfies 
$b_K\in \mathbb{P}^2(K)$, $supp(b_K)\subseteq K$, $b_K=0 \mbox{ on } \partial K$, and $0\leq b_K\leq 1\mbox{ on } K $.
Similarly, $b_E\in \mathbb{P}^2(K)$, $supp(b_E)\subseteq 
\omega_E:=\left\{K'\in \mathcal{T}_h:  E\in\cE_h (K')\right\}$, 
$b_E=0\mbox{  on  } \partial K\smallsetminus E$ and $0\leq b_E\leq 1\mbox{ in } \omega_E$.
We also recall from \cite{verfurth:94a} that, given $k\in\mathbb{N}$, there exists an extension operator
$L: C(E)\longrightarrow C(K)$ that satisfies $L(p)\in \mathbb{P}^k(K)$ and $L(p)_{|E}=p, \forall p\in \mathbb{P}^k(E)$.
A corresponding vectorial version of $L$, that is, the componentwise application of $L$, is denoted by 
$\textbf{L}$. Additional properties of $b_K$, $b_E$ and $L$ are collected in the following lemma (see \cite{verfurth:94a}):

\begin{lem}
	Given $k\in \mathbb{N}^*$, there exist positive constants depending only on $k$ and shape-regularity of the triangulations 
	(minimum angle condition), such that for each simplexe $K$ and $E\in \cE_h(K)$ there hold
	\begin{eqnarray}\label{cl1}
	\parallel q \parallel_{K}&\lesssim&\parallel qb_K^{1/2}\parallel_{K}\lesssim
	\parallel q\parallel_{K}, \forall q\in \mathbb{P}^k(K)\\\label{cl2}
	|q b_K|_{1,K}&\lesssim&  h_K^{-1}\parallel q \parallel_{K}, \forall q\in \mathbb{P}^k(K)\\\label{cl3}
	\parallel p\parallel_{E}&\lesssim&\parallel b_E^{1/2}p\parallel_{E}\lesssim \parallel p\parallel_{E},
	\forall p\in \mathbb{P}^k(E)\\\label{cl4}
	\parallel L(p)\parallel_{K} +h_E|L(p)|_{1,K}&\lesssim& h_E^{1/2}\parallel p\parallel_{E}
	\forall p\in \mathbb{P}^k(E)
	\end{eqnarray}
\end{lem}

To this end, we recall
some notation.
We define the error respect to  $\sigma$, $u$ and $\lambda$ respectively by 
$e_{\sigma}=\sigma- \sigma_h$, $e_u=u-u_h$ and $\mathbf{e}_{\lambda}=\lambda-\lambda_h$. Then we recall the global
error  define by 
$ \lVert \mathbf{U}-\mathbf{U}_h \lVert_{\textbf{H}}:=  \left(\lVert e_{\sigma}\lVert_{1,\Omega}^2+\lVert e_u \lVert_{2,\Omega}^2 +\lVert e_{\lambda} \lVert_{1/2,\partial\Omega}^2\right)^{1/2}$. To prove local efficiency for $w\subset \Omega$, let us denote by: 
\begin{eqnarray*}
\parallel\textbf{V}\parallel_{h,w}:=\left[\displaystyle\sum_{K\subset \bar{w}}\left(
\parallel v\parallel_{2,K}^2+\parallel\tau\parallel_{0,K}^2+\displaystyle\sum_{E\in\cE_h(\bar{K})}
\parallel\mu\parallel_{0,E}^2\right)\right]^{1/2}, \mbox{ where } \textbf{V}=(\tau,v,\mu)\in\textbf{H}.
\end{eqnarray*}
The main result of this subsection can be given as follows:
\begin{thm}
	Let $f\in L^2(\Omega)$, $g\in L^2(\partial \Omega)$. Let $(\sigma, u, \lambda)$ be the unique solution
	of the continous problem and $(\sigma_h, u_h, \lambda_h)$ be the unique solution of discret  
	problem. Then, the local error estimator $\Theta_K$ satisfies :
	\begin{equation}\label{lower}
	\Theta_K \lesssim \rVert (e_{\sigma}, e_u , e_{\lambda})\rVert_{h, \tilde{w}_K}+  
	\displaystyle \sum_{K' \subset \tilde{w}_K }\zeta_K', \quad \forall K\in \mathcal{T}_h.
	\end{equation}
	Where $\tilde{w}_K$ is a finite union of neighbouring elements of $K$.
\end{thm}

\begin{proof}
 To establish the lower error bound (\ref{lower}), we will make extensive
use of the original system of equations given by (\ref{m1}) and (\ref{m100}), which is
recovered from the mixed formulation (\ref{m11}) by choosing suitable test
functions and integrating by parts backwardly the corresponding
equations. Thereby, we bound each term of the residual separately.
\begin{enumerate}

	\item \textbf{Residual element} \quad  $(f_h-A:D^2u_h ):$ 
	
	Let us define by  $v_K= (f_h-A:D^2u_h)b_K$ where $b_K$ is the bubble fonction define in the Section \ref{efficiencysection}.
	We have 
	\begin{equation}
	(f_h-A:D^2u_h, v_K)_K = \int_{K} (f_h-A:D^2u_h).v_K
	\end{equation}
	Introduce $f$ and use the modified continuous formulation (\ref{m100})  to get:
	\begin{align*}
	(f_h-A:D^2u_h, v_K)_K = \int_{K} (f-f_h).v_K -\int_{K} (A:D^2u_h).v_K +  \int_{K} (\det \sigma).v_K  \\
	= \int_{K} (f-f_h).v_K + \int_{K} ((A:D^2u)-\int_{K} (A:D^2u_h)).v_K  \\ 
	= \int_{K} (f-f_h).v_K + \int_{K} A:((D^2u)-(D^2u_h)).v_K  \\ 
	\end{align*}
	Using Cauchy-Schwarz inequality we obtain
	\begin{equation}
	\rVert (f_h-A:D^2u_h) v_K^{\frac{1}{2}}\rVert \leq 
	C(\left(\rVert f-f_h \rVert_K \rVert v_K\rVert+ \rVert e_{u} \rVert_{2,K} \rVert v_K\rVert_K\right).
	\end{equation}
	Using an  inverses inequalities we have
	\begin{equation}
	\rVert f_h-A:D^2u_h \rVert_K \lesssim   \rVert e_{u} \rVert_K+ \zeta_K.
	\end{equation}
	Thus,
	\begin{eqnarray}\label{ind1}
	\rVert f_h-A:D^2u_h \rVert_K	&\lesssim&
		\rVert (e_{\sigma}, e_u , e_{\lambda})\rVert_{h, \tilde{w}_K}+
		\zeta_K.
	\end{eqnarray}
	 \item \textbf{Residual element} \quad  $(\sigma_h-D^2u_h ):$  Let $K\in \cT_h$. We have
	 \begin{eqnarray*}
	 	\sigma_h-D^2u_h&=&(\sigma_h-\sigma)+(\sigma-D^2u_h)\\
	 	&=& \sigma_h-\sigma+(D^2u-D^2u_h)\\
	 	&=& e_{\sigma}+(D^2e_u)
	 \end{eqnarray*}
 The triangular inequality leads to
 \begin{eqnarray*}
 	\parallel \sigma_h-D^2u_h\parallel_K&\lesssim& \parallel e_{\sigma}\parallel_K+\parallel e_u\parallel_{2,K}
 \end{eqnarray*}
	Hence, 
	\begin{eqnarray}\label{ind2}
		\parallel \sigma_h-D^2u_h\parallel_K&\lesssim& 
		\rVert (e_{\sigma}, e_u , e_{\lambda})\rVert_{h, \tilde{w}_K}+
		\zeta_K.
	\end{eqnarray}

	\item 
	\textbf{Residual element} \quad  $(f_h+\det \sigma_h-A:D^2u_h)$: From the  residu equation of differential form (\ref{residuequation}), we obtain for $\textbf{W}=(\tau,v,\mu)$:
	$$
	(f-\det \sigma_h-A:D^2u_h,v)=\langle D_{\textbf{U}}\textbf{B}(\textbf{U}-\textbf{U}_h),\textbf{W}\rangle-R(\textbf{W})
	$$
	and we deduce for $\textbf{W}=(0,v,0)$:
	\begin{eqnarray*}
	(f_h-\det \sigma_h-A:D^2u_h,v)_K&=&-(A:D^2 u,v)+(A:D^2 u_h,v)_K-R(\textbf{W})-(f-f_h,v)_K	\\
	&=&-(A:D^2 (u-u_h),v)_K-(f-A:D^2u_h,v)_K-(f-f_h,v)_K\\
	&=&-(A:D^2 e_u,v)_K-(f_h-A:D^2u_h,v)_K-2(f-f_h,v)_K
\end{eqnarray*}
	Using Cauchy Schwarz inequality, we obtain
	\begin{eqnarray*}
		|(f_h-\det \sigma_h-A:D^2u_h,v)_K|
			&\lesssim&\left(|e_u|_{2,K}+\parallel f_h-A:D^2u_h\parallel_K+\parallel f-f_h\parallel_K\right)\parallel v\parallel_K
	\end{eqnarray*}
Hence, from (\ref{ind1}), we have,
\begin{eqnarray*}
	\displaystyle\sup_{v\in H^1(K)}\frac{|(f_h-\det \sigma_h-A:D^2u_h,v)_K|}{\parallel v\parallel_K}&\lesssim& 
	|e_u|_{2,K}+\parallel e_u\parallel_K+\zeta_K\\
	&\lesssim& \parallel e_u\parallel_{2,K}+\zeta_K.
\end{eqnarray*}
	Thus,
	\begin{eqnarray}\label{ind3}
		\parallel f_h-\det \sigma_h-A:D^2 u_h\parallel_K&\lesssim& \parallel e_u\parallel_{2,K}+\zeta_K.
	\end{eqnarray}

	\item \textbf{Residual element } $(g_h-\lambda_h)$: For each $K\in\cT_h$ and $E\in\cE_h(K\cap \partial \Omega)$, we have
	\begin{eqnarray*}
		g_h-\lambda_h&=&(g_h-\lambda)+(\lambda-\lambda_h)\\
		&=&(g_h-g)+e_{\lambda}
	\end{eqnarray*}
Thus, 
\begin{eqnarray*}
	\parallel g_h-\lambda_h\parallel_E &\leq& \parallel g_h-g\parallel_E+\parallel e_{\lambda}\parallel_E
\end{eqnarray*}
We sum up over all (boundary) faces $E\in\cE_h(K\cap \partial \Omega)$ and obtain,
\begin{eqnarray}\label{ind4}
	\displaystyle\sum_{E\in\cE_h(K\cap\partial \Omega)}\parallel g_h-\lambda_h\parallel_E &\lesssim&
	\rVert (e_{\sigma}, e_u , e_{\lambda})\rVert_{h, \tilde{w}_K}+
	 \zeta_K.
\end{eqnarray}
	\end{enumerate}
The estimates (\ref{ind1}), (\ref{ind2}), (\ref{ind3}) and (\ref{ind4}) provide the desired
local lower error bound.
\end{proof}
\section{Summary}\label{summary}
In this paper, we have proposed and rigorously analyzed a new a posteriori residual type error estimators for the Monge-Amp\`ere equation on isotropic meshes. Our investigations cover conforming
discretization in $\mathbb{R}^2$. The residual type a posteriori error estimator is provided. It is proven that the a posteriori error estimate provided in this paper is both reliable and efficient. Many issues remain to be addressed in this area, let us mention other types of a posteriori error estimators or implementation and convergence analysis of adaptive finite element methods.
Also, we intend use  the technical used by \cite{Jad:2014} where the adaptative technical was used 
for a nonlinear problem. Ideed \cite{Jad:2014} applicate   Brezzi-Rappaz-Raviart theorem and this 
permit to have the implicitly   inf-sup condition. In the same way w'll  suppose this 
implicitly   inf-sup condition for Monge-Amp\`ere problem in order to built our error indicator.

\end{Large}
\end{document}